\documentclass[final,3p]{elsarticle}
\usepackage{lineno,hyperref}
\usepackage{amsmath}
\usepackage{amssymb}
\usepackage[linesnumbered,ruled,vlined]{algorithm2e}
\usepackage{caption}
\usepackage{mathtools}
\usepackage{changes}
\usepackage{multirow}
\usepackage{verbatim}
\usepackage{mathrsfs}
\usepackage{graphicx}
\usepackage{subcaption}
\usepackage{amsmath,amsthm,bm,mathrsfs}
\theoremstyle{plain}% Theorem-like structures provided by amsthm.sty
\newtheorem{theorem}{Theorem}[section]

\theoremstyle{definition}

\theoremstyle{remark}

\modulolinenumbers[5]

\journal{ArXiv.org}

\begin{document}

\begin{frontmatter}

\title{D-RADI: A Low-rank ADI Algorithm for Solving Large-scale Discrete-time Algebraic Riccati Equations}

\author[uz]{Umair~Zulfiqar\corref{mycorrespondingauthor}}
\cortext[mycorrespondingauthor]{Corresponding author}
\ead{umair@yangtzeu.edu.cn}
\address[uz]{School of Electronic Information and Electrical Engineering, Yangtze University, Jingzhou, Hubei, 434023, China}

\begin{abstract}
The low-rank alternating direction implicit (ADI) method is an efficient numerical technique for solving several types of large-scale matrix equations that admit low-rank solutions. The discrete-time algebraic Riccati equation (DARE) is an important matrix equation with applications in state estimation, controller design, and filter design. In the literature, the low-rank Cholesky factor ADI method for Stein equations has been used within Newton iterations to solve large-scale DAREs. However, no dedicated low-rank ADI solver is available for such DAREs. To address this gap, this paper presents a low-rank ADI solver for large-scale DAREs. We also propose an efficient approach to generate ADI shifts automatically, which makes the proposed solver fully autonomous for solving DAREs. The effectiveness of the proposed solver is compared with MATLAB's \texttt{idare} on a moderate-order problem. Efficiency and accuracy are further demonstrated on large-scale DAREs of order $10^6$. Numerical results confirm that the solver is efficient, accurate, and fully autonomous.
\end{abstract}

\begin{keyword}
ADI\sep Discrete-time\sep Low-rank\sep Projection\sep Rational interpolation\sep Riccati equation
\end{keyword}

\end{frontmatter}

%\linenumbers
\section{Introduction}
We consider the following discrete-time algebraic Riccati equation (DARE):
\begin{align}
A^\top Q A-E^\top Q E-\big(A^\top QB+C_2^\top\big)(B^\top Q B+R)^{-1}\big(A^\top QB+C_2^\top\big)^\top+C_1^\top ZC_1&=0,\label{dare}
\end{align}
where $Q=Q^\top\in\mathbb{R}^{n\times n}$, $E\in\mathbb{R}^{n\times n}$, $A\in\mathbb{R}^{n\times n}$, $B\in\mathbb{R}^{n\times m}$, $C_1\in\mathbb{R}^{p\times n}$, $C_2\in\mathbb{R}^{m\times n}$, $R=R^\top\in\mathbb{R}^{m\times m}$, and $Z=Z^\top\in\mathbb{R}^{p\times p}$. Assume $E$ and $R$ are invertible, and $R$ and $Z$ are indefinite.

Define the state-feedback gain $K$ as
\[
K=\big(B^\top Q B+R\big)^{-1}\big(B^\top Q A+C_2\big).
\]
Next, define the closed-loop state matrix $A_{cl}$ as
\[
A_{cl}=A-BK.
\]
Then $Q_\star$ is a stabilizing solution to the DARE \eqref{dare} if the matrix $E^{-1}A_{cl}$ is Schur stable, i.e.,
\[
|\lambda(E^{-1}A_{cl})|<1.
\]
We focus on computing the stabilizing solution to the DARE \eqref{dare}.

Let us define $\hat{A}$, $\hat{C}$, and $\hat{Z}$ as follows:
\[
\hat{A}=A-BR^{-1}C_2,\quad \hat{C}=\begin{bmatrix}C_1\\C_2\end{bmatrix},\quad \hat{Z}=\begin{bmatrix}Z&0\\0&-R^{-1}\end{bmatrix}.
\]
Then the DARE \eqref{dare} can be rewritten as
\[
\hat{A}^\top Q \hat{A}-E^\top Q E-\hat{A}^\top Q B(B^\top Q B+R)^{-1}B^\top Q\hat{A}+\hat{C}^\top\hat{Z}\hat{C}=0.
\]
When $m,p\ll n$, the solution to the DARE \eqref{dare} is typically numerically low-rank, which enables low-rank solution methods even when $n$ is large and direct methods are prohibitively expensive. Note that this is the DARE form solved by MATLAB’s \texttt{idare} and \texttt{dare}, but the applicability of these routines are limited to moderate $n$.
\section{Main Work}
Low-rank ADI methods are efficient solvers for large-scale matrix equations that admit low-rank solutions \cite{wachspress1988iterative,simoncini2016computational,benner2013numerical}. They have been extended to several types of linear and nonlinear matrix equations, including Lyapunov equations \cite{benner2013efficient,benner2013reformulated,zulfiqar2026new}, Sylvester equations \cite{benner2009adi,benner2014computing}, Stein equations \cite{benner2011numerical,zulfiqar2026}, continuous-time algebraic Riccati equations (CAREs) \cite{benner2018radi,bertram2024family,zulfiqar2025unified,zulfiqar2026ldl}, and nonsymmetric continuous-time algebraic Riccati equations (NCAREs) \cite{zulfiqar2026low}. However, to the best of our knowledge, no low-rank ADI algorithm exists for the DARE of the form \eqref{dare}. The low-rank Cholesky-factor ADI algorithm for Stein equations has been used within Newton iterations to compute the low-rank solution of a special case of DARE in \cite{benner2011numerical}. Unlike CAREs, however, DAREs of the form \eqref{dare} lack a low-rank ADI algorithm analogous to the Riccati ADI method (RADI) \cite{benner2018radi}. In this paper, we propose a discrete-time counterpart of RADI \cite{benner2018radi}, which produces a low-rank solution to the DARE \eqref{dare} in the form $Q\approx W_r^{(k)} Q_r^{(k)} (W_r^{(k)})^\top$, where $W_r^{(k)}\in\mathbb{R}^{n\times r}$, $Q_r^{(k)}\in\mathbb{R}^{r\times r}$, and $r=k(p+m)\ll n$.

As shown in our previous works \cite{zulfiqar2025unified,zulfiqar2026ldl,zulfiqar2026low,zulfiqar2026new,zulfiqar2026}, despite their different origins and developments, low-rank ADI methods are essentially Petrov–Galerkin projection-based recursive rational interpolation algorithms that interpolate at the mirror images of the ADI shifts. Various solvers differ mainly in their pole-placement properties, as elaborated below.

Let $\{\alpha_i\}_{i=1}^{k}$ be the ADI shifts used to obtain a low-rank approximation $Q \approx \hat{Q}^{(k)} = W_r^{(k)} Q_r^{(k)} (W_r^{(k)})^\top$, with $W_r^{(k)}$ satisfying
\begin{align}
\operatorname{span}_{i=1,\dots,k}\left\{(-\alpha_i E^\top - \hat{A}^\top)^{-1} \hat{C}^\top\right\} \subset \mathrm{Ran}(W_r^{(k)}). \label{int_prop}
\end{align}
Define the residual $R_q^{(k)}$ for the approximation $Q \approx \hat{Q}^{(k)}$ as
\[
R_q^{(k)} = \hat{A}^\top \hat{Q}^{(k)} \hat{A}-E^\top \hat{Q}^{(k)} E-\hat{A}^\top \hat{Q}^{(k)} B(B^\top \hat{Q}^{(k)} B+R)^{-1}B^\top \hat{Q}^{(k)}\hat{A}+\hat{C}^\top\hat{Z}\hat{C}.
\]
For some generally unknown matrix $V_r^{(k)}$ satisfying $(W_r^{(k)})^\top E V_r^{(k)} = I$, the residual of a general low-rank ADI method satisfies the Petrov–Galerkin projection condition $(V_r^{(k)})^\top R_q^{(k)} V_r^{(k)} = 0$. In other words, $Q_r^{(k)}$ is the solution to the following projected DARE:
\begin{align}
(\hat{A}_r^{(k)})^\top Q_r^{(k)} \hat{A}_r^{(k)}- Q_r^{(k)}-(\hat{A}_r^{(k)})^\top Q_r^{(k)} B_r^{(k)}\Big((B_r^{(k)})^\top Q_r^{(k)} B_r^{(k)}+R\Big)^{-1}(B_r^{(k)})^\top Q_r^{(k)}\hat{A}_r^{(k)}+(\hat{C}_r^{(k)})^\top\hat{Z}\hat{C}_r^{(k)}=0,\label{proj_dare}
\end{align}
where
\begin{align}
\hat{A}_r^{(k)}&=(W_r^{(k)})^\top \hat{A} V_r^{(k)}=A_r^{(k)}-B_r^{(k)}R^{-1}C_{2,r}^{(k)},\quad B_r^{(k)}=(W_r^{(k)})^\top B,\quad \hat{C}_r^{(k)}=\begin{bmatrix}C_{1,r}^{(k)}\\C_{2,r}^{(k)}\end{bmatrix}=\hat{C}V_r^{(k)}.
\end{align}
Define $G(z)$ and $G_r^{(k)}(z)$ as follows:
\[
G(z)=\hat{C}(zE-\hat{A})^{-1}B\quad \text{and}\quad G_r^{(k)}(z)=\hat{C}_r^{(k)}(zI-\hat{A}_r^{(k)})^{-1}B_r^{(k)},
\]
where $z=e^{j\omega}$.

Property \eqref{int_prop} implies the interpolation conditions
\begin{align}
G(-\alpha_i)=G_r^{(k)}(-\alpha_i)\label{int_cond}
\end{align}
for $i=1,\dots,k$ \cite{beattie2017chapter}.

Define the projected state-feedback gain $K_r^{(k)}$ as
\begin{align}
K_r^{(k)}=\big((B_r^{(k)})^\top Q_r^{(k)} B_r^{(k)}+R\big)^{-1}\big((B_r^{(k)})^\top Q_r^{(k)} A_r^{(k)}+C_{2,r}^{(k)}\big).
\end{align}
Define the projected closed-loop state matrix $A_{r,cl}^{(k)}$ as
\begin{align}
A_{r,cl}^{(k)}=A_r^{(k)}-B_r^{(k)}K_r^{(k)}.
\end{align}
The desired pole-placement property of the proposed low-rank ADI solver for the DARE \eqref{dare} is that the poles of $A_{r,cl}^{(k)}$ lie within the unit circle, ensuring that $Q_r^{(k)}$ is a stabilizing solution to the projected DARE \eqref{proj_dare}.

As noted in \cite{zulfiqar2025unified}, another key feature of existing low-rank ADI methods is that the solutions to the projected matrix equations are block diagonal, with each block corresponding to an ADI shift $\alpha_i$. Each diagonal block can be computed from analytical formulas, so the projected matrix equations are only implicitly solved. Existing low-rank ADI algorithms recursively accumulate the low-rank solution without updating the previous solution. The main computational ingredients are shifted linear solves and basic matrix operations. Next, we develop a low-rank ADI solver for the DARE \eqref{dare} that has all these properties.
\subsection{A Low-rank Solver for DARE \eqref{dare}}
Assume complex-valued ADI shifts occur in conjugate pairs, i.e., $\alpha_{i+1} = \overline{\alpha_i}$ whenever $\mathrm{Im}(\alpha_i) \neq 0$. For each ADI shift $\alpha_i$, define $s_r^{(i)}$ and $l_r^{(i)}$ as follows:
\begin{align}
s_r^{(i)} & = 
\begin{cases} 
-\alpha_i I_{pm}, \hspace*{5cm} \text{if } \mathrm{Im}(\alpha_i) = 0, \\
\begin{bmatrix} 
-\mathrm{Re}(\alpha_i)I_{pm} & -\mathrm{Im}(\alpha_i)I_{pm} \\ 
\mathrm{Im}(\alpha_i)I_{pm} & -\mathrm{Re}(\alpha_i) I_{pm}
\end{bmatrix},  \hspace*{1.6cm}\text{if } \mathrm{Im}(\alpha_i) \neq 0,
\end{cases}\label{sw}\\
l_r^{(i)} &= 
\begin{cases} 
-I_{pm}, \hspace*{5.4cm} \text{if } \mathrm{Im}(\alpha_i) = 0, \\
\begin{bmatrix} 
-I_{pm} & 0
\end{bmatrix}, \hspace*{4.5cm} \text{if } \mathrm{Im}(\alpha_i) \neq 0,
\end{cases}\label{lw}
\end{align}
where $I_{pm}$ denotes the identity matrix of size $(p+m) \times (p+m)$.

Define $S_r^{(i)}$, $(S_r^{(i)})^{-1}$, $L_r^{(i)}$, $X_r^{(i)}$, $W_r^{(i)}$, and $\hat{C}_{\perp}^{(i)}$ as follows:
\begin{align}
S_r^{(i)}&=\begin{bmatrix}S_r^{(i-1)}&S_{12}^{(i)}\\\mathbf{0}&s_r^{(i)}\end{bmatrix},\quad (S_r^{(i)})^{-1}=\begin{bmatrix}(S_r^{(i-1)})^{-1}&-(S_r^{(i-1)})^{-1}S_{12}^{(i)}(s_r^{(i)})^{-1}\\\mathbf{0}&(s_r^{(i)})^{-1}\end{bmatrix},\quad L_r^{(i)}=\begin{bmatrix}L_r^{(i-1)}&l_r^{(i)}\end{bmatrix},\nonumber\\
X_r^{(i)}&=\begin{bmatrix}X_r^{(i-1)}&\mathbf{0}\\\mathbf{0}&x_r^{(i)}\end{bmatrix},\quad W_r^{(i)}=\begin{bmatrix}W_r^{(i-1)}&w_r^{(i)}\end{bmatrix},\quad \hat{C}_{\perp}^{(i)}=\hat{C}-L_r^{(i)}(S_r^{(i)})^{-1}(X_r^{(i)})^{-1}(W_r^{(i)})^\top E,
\end{align}
where
\begin{align}
S_{12}^{(i)}=(X_r^{(i-1)})^{-1}(S_r^{(i-1)})^{-\top}\Big((L_r^{(i-1)})^\top\hat{Z}l_r^{(i)}+(W_r^{(i-1)})^\top BR^{-1}B^\top w_r^{(i)}\Big),\\
(s_r^{(i)})^\top x_r^{(i)} s_r^{(i)}- x_r^{(i)}-(l_r^{(i)})^\top \hat{Z}l_r^{(i)}-(w_r^{(i)})^\top BR^{-1}B^\top w_r^{(i)}+(S_{12}^{(i)})^\top X_r^{(i-1)}S_{12}^{(i)}=0,\label{small_x_r}\\
\Big(\hat{A}^\top-E^\top W_r^{(i-1)}(X_r^{(i-1)})^{-1}(S_r^{(i-1)})^{-\top}(W_r^{(i-1)})^\top BR^{-1}B^\top\Big) w_r^{(i)}-E^\top w_r^{(i)}s_r^{(i)}+(\hat{C}_{\perp}^{(i-1)})^\top \hat{Z}l_r^{(i)}=0.\label{small_w_r}
\end{align}
By substitution, $W_r^{(i)}$ and $X_r^{(i)}$ solve the following equations:
\begin{align}
\hat{A}^\top W_r^{(i)}-E^\top W_r^{(i)}S_r^{(i)}+\hat{C}^\top \hat{Z}L_r^{(i)}&=0,\label{sylv_W}\\
(S_r^{(i)})^\top X_r^{(i)}S_r^{(i)}-X_r^{(i)}-(L_r^{(i)})^\top \hat{Z} L_r^{(i)}-(W_r^{(i)})^\top BR^{-1}B^\top W_r^{(i)}&=0.\label{lyap_x}
\end{align}
Moreover, $\hat{C}_{\perp}^{(i)}$ can be recursively computed as
\begin{align}
\hat{C}_{\perp}^{(i)}=\hat{C}_{\perp}^{(i-1)}-\Bigg(\Big(l_r^{(i)}-L_r^{(i-1)}(S_r^{(i-1)})^{-1}S_{12}^{(i)}\Big)(s_r^{(i)})^{-1}(x_r^{(i)})^{-1}\Bigg)(w_r^{(i)})^\top E.\label{C_perp}
\end{align}
Due to the block triangular structure of $S_r^{(i)}$, its eigenvalues are $-\alpha_1,\dots,-\alpha_i$, each with multiplicity $p+m$. By the connection between Sylvester equations and rational Krylov subspaces established in \cite{gallivan2004sylvester}, $W_r^{(i)}$ satisfies property \eqref{int_prop}.

Let $V_r^{(i)}$ satisfy $(W_r^{(i)})^\top E V_r^{(i)} = I$. Pre-multiplying \eqref{sylv_W} by $V_r^{(i)}$ gives
\begin{align}
(\hat{A}_r^{(i)})^\top-S_r^{(i)}+(\hat{C}_r^{(i)})^\top\hat{Z}L_r^{(i)}=0.\nonumber
\end{align}
Consequently, $\hat{A}_r^{(i)} = (S_r^{(i)})^\top - (L_r^{(i)})^\top \hat{Z} \hat{C}_r^{(i)}$ can be parameterized in terms of $\hat{C}_r^{(i)}$. If the pair $(S_r^{(i)},\hat{Z}L_r^{(i)})$ is observable, this parameterization does not affect the interpolation condition \eqref{int_cond}, since it is equivalent to varying $V_r^{(i)}$; cf. \cite{wolfthesis,panzerthesis,astolfi2010model}. The following theorem shows that a specific choice of $\hat{C}_r^{(i)}$ guarantees that $Q_r^{(i)}=(X_r^{(i)})^{-1}$ is a stabilizing solution to the projected DARE \eqref{proj_dare}.

\begin{theorem}\label{th1}
Let the ADI shifts $\alpha_i$ lie outside the unit circle, i.e., $|\alpha_i|>1$. Let $W_r^{(i)}$ solve the Sylvester equation \eqref{sylv_W} with $s_r^{(i)}$, $l_r^{(i)}$, $x_r^{(i)}$, $w_r^{(i)}$, $S_r^{(i)}$, $L_r^{(i)}$, $S_{12}^{(i)}$, $\hat{C}_{\perp}^{(i)}$, and $X_r^{(i)}$ defined in \eqref{sw}--\eqref{C_perp}. Assume there exists a matrix $V_r^{(i)}$ satisfying $(W_r^{(i)})^\top E V_r^{(i)}=I$, so that $\hat{A}_r^{(i)} = (W_r^{(i)})^\top \hat{A} V_r^{(i)} = (S_r^{(i)})^\top - (L_r^{(i)})^\top \hat{Z} \hat{C}_r^{(i)}$, $B_r^{(i)}=(W_r^{(i)})^\top B$, and $\hat{C}_r^{(i)} =\hat{C}V_r^{(i)}$. Assume that the pair $(S_r^{(i)},\hat{Z}L_r^{(i)})$ is observable, and that the matrices $X_r^{(i)}$ and $R+(B_r^{(i)})^\top (X_r^{(i)})^{-1} B_r^{(i)}$ are invertible. If the free parameter $\hat{C}_r^{(i)}$ is chosen as
\[
\hat{C}_r^{(i)} =\begin{bmatrix}C_{1,r}^{(i)}\\C_{2,r}^{(i)}\end{bmatrix}= \begin{bmatrix} \hat{c}_r^{(1)} & \cdots & \hat{c}_r^{(i)} \end{bmatrix} = L_r^{(i)}\big( S_r^{(i)}\big)^{-1}\big(X_r^{(i)}\big)^{-1},
\]
with $\hat{c}_r^{(i)}=\Big(l_r^{(i)}-L_r^{(i-1)}\big(S_r^{(i-1)}\big)^{-1}S_{12}^{(i)}\Big)(s_r^{(i)})^{-1}(x_r^{(i)})^{-1}$,
the following statements hold:
\begin{enumerate}
  \item $Q_r^{(i)}=(X_r^{(i)})^{-1}$ is a stabilizing solution to the projected DARE \eqref{proj_dare}. The closed-loop matrix
  \begin{align}
  A_{r,cl}^{(i)}=A_r^{(i)}-B_r^{(i)}\Big((B_r^{(i)})^\top Q_r^{(i)} B_r^{(i)}+R\Big)^{-1}\Big((B_r^{(i)})^\top Q_r^{(i)} A_r^{(i)}+C_{2,r}^{(i)}\Big),
  \end{align}
  has eigenvalues $-\frac{1}{\alpha_1},\dots,-\frac{1}{\alpha_i}$, each with multiplicity $p+m$.
  \item The residual $R_q^{(i)}$, which satisfies the Petrov–Galerkin projection condition $(V_r^{(i)})^\top R_q^{(i)} V_r^{(i)} = 0$, is given by
  \[
  R_q^{(i)}=\big(\hat{C}_{\perp}^{(i)}\big)^\top \hat{Z}\big(M^{(i)}\big)^{-1} \hat{C}_{\perp}^{(i)},
  \]
  where $M^{(i)}=I_{pm}-\hat{C}_r^{(i)}X_r^{(i)}\big(\hat{C}_r^{(i)}\big)^\top\hat{Z}=M^{(i-1)}-\hat{c}_r^{(i)}x_r^{(i)}(\hat{c}_r^{(i)})^\top\hat{Z}$ with $M^{(0)}=I_{pm}$.
  \item $W_r^{(i)}$ solves the Sylvester equation
  \begin{align}
  \hat{A}^\top W_r^{(i)} - E^\top W_r^{(i)} (\hat{A}_r^{(i)})^\top + (\hat{C}_{\perp}^{(i)})^\top \hat{Z} L_r^{(i)} = 0.\label{radi_sylv2}
  \end{align}
\end{enumerate}
\end{theorem}

\begin{proof}
The proof is given in the appendix.
\end{proof}
\subsection{Algorithm}
The low-rank solver developed in the previous subsection is recursive and thus has all the features of a low-rank ADI solver mentioned earlier in this section. It is straightforward to show that when $B=\mathbf{0}$, $R=I$, and $C_2=\mathbf{0}$, the DARE \eqref{dare} reduces to a Stein equation. In this case, the proposed solver reduces to the low-rank ADI solver for Stein equations proposed in \cite{benner2011numerical} and \cite{benner2014computing}, and Theorem \ref{th1} reduces to the results derived in Propositions 2.1 and 2.3 and Theorem 2.2 of \cite{zulfiqar2026}.

The matrix $w_r^{(i)}$ requires one shifted linear solve, while $x_r^{(i)}$ is obtained analytically, without solving the Stein equation \eqref{small_x_r}, as follows.

If $\operatorname{Im}(\alpha_i)=0$, then
\begin{equation}
x_r^{(i)}
=
\frac{
\hat{Z}
+
(w_r^{(i)})^\top B R^{-1} B^\top w_r^{(i)}
-
(S_{12}^{(i)})^\top X_r^{(i-1)} S_{12}^{(i)}
}{
\alpha_i^2-1
}.
\label{eq:x_real}
\end{equation}

If $\operatorname{Im}(\alpha_i)\neq 0$, decompose $w_r^{(i)}$ as $w_r^{(i)}=\begin{bmatrix}w_{\mathrm{re}}^{(i)}&w_{\mathrm{im}}^{(i)}\end{bmatrix}$, define
\[
b_{\mathrm{re}}^{(i)}
:=
(w_{\mathrm{re}}^{(i)})^\top B,
\qquad
b_{\mathrm{im}}^{(i)}
:=
(w_{\mathrm{im}}^{(i)})^\top B,
\]
and partition
\[
(S_{12}^{(i)})^\top X_r^{(i-1)} S_{12}^{(i)}
=
\begin{bmatrix}
Y_{11}^{(i)} & Y_{12}^{(i)}\\
Y_{21}^{(i)} & Y_{22}^{(i)}
\end{bmatrix},
\]
where each block has size $(p+m)\times(p+m)$. Set
\begin{equation}
\begin{aligned}
N_{11}^{(i)}
&:=
Y_{11}^{(i)}
-
\hat{Z}
-
b_{\mathrm{re}}^{(i)}R^{-1}(b_{\mathrm{re}}^{(i)})^\top,\\
N_{12}^{(i)}
&:=
Y_{12}^{(i)}
-
b_{\mathrm{re}}^{(i)}R^{-1}(b_{\mathrm{im}}^{(i)})^\top,\\
N_{21}^{(i)}
&:=
Y_{21}^{(i)}
-
b_{\mathrm{im}}^{(i)}R^{-1}(b_{\mathrm{re}}^{(i)})^\top,\\
N_{22}^{(i)}
&:=
Y_{22}^{(i)}
-
b_{\mathrm{im}}^{(i)}R^{-1}(b_{\mathrm{im}}^{(i)})^\top.
\end{aligned}
\label{eq:x_complex_blocks}
\end{equation}
Define the complex-valued matrices
\begin{equation}
\Xi^{(i)}
:=
-\frac{
N_{11}^{(i)}+N_{22}^{(i)}
+
\mathrm{j}\bigl(N_{21}^{(i)}-N_{12}^{(i)}\bigr)
}{
|\alpha_i|^2-1
},
\qquad
\Upsilon^{(i)}
:=
-\frac{
N_{11}^{(i)}-N_{22}^{(i)}
+
\mathrm{j}\bigl(N_{12}^{(i)}+N_{21}^{(i)}\bigr)
}{
\alpha_i^2-1
}.
\label{eq:x_complex_aux}
\end{equation}
Then
\begin{equation}
x_r^{(i)}
=
\frac{1}{2}
\begin{bmatrix}
\operatorname{Re}(\Xi^{(i)})+\operatorname{Re}(\Upsilon^{(i)})
&
\operatorname{Im}(\Upsilon^{(i)})-\operatorname{Im}(\Xi^{(i)})
\\[2mm]
\operatorname{Im}(\Upsilon^{(i)})+\operatorname{Im}(\Xi^{(i)})
&
\operatorname{Re}(\Xi^{(i)})-\operatorname{Re}(\Upsilon^{(i)})
\end{bmatrix}.
\label{eq:x_complex}
\end{equation}

These formulas yield a recursive implementation of the solver developed in the previous subsection. The pseudo-code of the resulting solver, named Discrete-time RADI (D-RADI), is given in Algorithm \ref{alg1}. The main computational step is the shifted linear solve in Step \ref{step1}, which can be performed efficiently using the Sherman–Morrison–Woodbury (SMW) formula \cite{golub2013matrix}, as done in the continuous-time counterparts \cite{benner2018radi} and \cite{zulfiqar2026ldl}. A MATLAB-based implementation of D-RADI is publicly available at \cite{mycode}.
\begin{algorithm}[!t]
\caption{D-RADI}\label{alg1}
\DontPrintSemicolon
\KwIn{
  Matrices of DARE \eqref{dare}: $E$, $A$, $B$, $R$, $C_1$, $Z$, $C_2$; ADI shifts: $\{\alpha_i\}_{i=1}^k$ satisfying $|\alpha_i|>1$; Tolerance: $\tau\in[0,1]$.
}
\KwOut{
  Approximation of $Q$: $Q \approx W_r^{(i)}Q_r^{(i)}(W_r^{(i)})^\top$; Approximation of gain matrix $K$: $K\approx \hat{K}^{(i)}=\big(B^\top W_r^{(i)}Q_r^{(i)}(W_r^{(i)})^\top B+R\big)^{-1}\big(B^\top W_r^{(i)}Q_r^{(i)}(W_r^{(i)})^\top A+C_2\big)$; Residual $R_q^{(i)}$: $R_q^{(i)}=\big(\hat{C}_{\perp}^{(i)}\big)^\top \hat{Z}\big(M^{(i)}\big)^{-1} \hat{C}_{\perp}^{(i)}$.
}

\BlankLine
\textbf{Initialization:}
$\hat{A}=A-BR^{-1}C_2$, $\hat{C} = \begin{bmatrix}C_1 \\ C_2\end{bmatrix}$, $\hat{Z} = \mathrm{blkdiag}(Z, -R^{-1})$, $\hat{C}_{\perp}^{(0)} = \hat{C}$, $W_r^{(0)} = [\;]$, $X_r^{(0)} = [\;]$, $Q_r^{(0)} = [\;]$, $S_r^{(0)}=[\;]$, $L_r^{(0)}=[\;]$, $\big(S_r^{(0)}\big)^{-1}=[\;]$, $S_{12}^{(0)}=[\;]$, $K_b^{(0)} = \mathbf{0}$, $K_1^{(0)}=R$, $K_2^{(0)}=C_2$, $M^{(0)}=I_{pm}$, $B_r^{(0)}=[\;]$, $\hat{C}_r^{(0)}=[\;]$, $i = 1$.

\BlankLine
\While{$\dfrac{\big\|\big(\hat{C}_{\perp}^{(i-1)}\big)^\top \hat{Z}\big(M^{(i-1)}\big)^{-1} \hat{C}_{\perp}^{(i-1)}\big\|}{\big\|\hat{C}^\top\hat{Z}\hat{C}\big\|} \geq \tau$}{
  Solve for $v_i$: $\big(\hat{A}^\top - (K_b^{(i-1)})^\top B^\top +\alpha_i E^\top \big) v_i = \big(\hat{C}_{\perp}^{(i-1)}\big)^\top$.\label{step1}
   
  \uIf{$\mathrm{Im}(\alpha_i) = 0$}{
    Set $w_r^{(i)} = v_i \hat{Z}$, $b_r^{(i)}=\big(w_r^{(i)}\big)^\top B$, $s_r^{(i)}=-\alpha_iI_{pm}$, $l_r^{(i)}=-I_{pm}$, and $S_{12}^{(i)}=\big(\hat{C}_r^{(i-1)}\big)^\top\hat{Z}l_r^{(i)}+Q_r^{(i-1)}\big(S_r^{(i-1)}\big)^{-\top}B_r^{(i-1)}R^{-1}\big(b_r^{(i)}\big)^\top$. \\
    Compute $x_r^{(i)}$ from \eqref{eq:x_real}, set $q_r^{(i)}=\big(x_r^{(i)}\big)^{-1}$, and $\hat{c}_r^{(i)}=\Big(l_r^{(i)}-L_r^{(i-1)}\big(S_r^{(i-1)}\big)^{-1}S_{12}^{(i)}\Big)(s_r^{(i)})^{-1}q_r^{(i)}$.\\
    Update $\hat{C}_{\perp}^{(i)} = \hat{C}_{\perp}^{(i-1)} - \hat{c}_r^{(i)}\big(w_r^{(i)}\big)^\top E$, $K_b^{(i)} = K_b^{(i-1)} + R^{-1}\Big(\big(b_r^{(i)}\big)^\top-\big(B_r^{(i-1)}\big)^\top \big(S_r^{(i-1)}\big)^{-1}S_{12}^{(i)}\Big)\big(s_r^{(i)}\big)^{-1}q_r^{(i)}\big(w_r^{(i)}\big)^\top E$, $K_1^{(i)}=K_1^{(i-1)}+\big(b_r^{(i)}\big)^\top q_r^{(i)}b_r^{(i)}$, $K_2^{(i)}=K_2^{(i-1)}+\big(b_r^{(i)}\big)^\top q_r^{(i)}\big(w_r^{(i)}\big)^\top A$, and $M^{(i)}=M^{(i-1)}-\hat{c}_r^{(i)}x_r^{(i)}\big(\hat{c}_r^{(i)}\big)^\top \hat{Z}$.\label{step7}\\
    Expand $X_r^{(i)} = \mathrm{blkdiag}\big(X_r^{(i-1)}, x_r^{(i)}\big)$, $Q_r^{(i)} = \mathrm{blkdiag}\big(Q_r^{(i-1)}, q_r^{(i)}\big)$, $W_r^{(i)} = \begin{bmatrix} W_r^{(i-1)} & w_r^{(i)}\end{bmatrix}$, $B_r^{(i)}=\begin{bmatrix}B_r^{(i-1)}\\b_r^{(i)}\end{bmatrix}$, $\hat{C}_r^{(i)}=\begin{bmatrix}\hat{C}_r^{(i-1)}&\hat{c}_r^{(i)}\end{bmatrix}$, $S_r^{(i)}=\begin{bmatrix}S_r^{(i-1)}&S_{12}^{(i)}\\\mathbf{0}&s_r^{(i)}\end{bmatrix}$, $\big(S_r^{(i)}\big)^{-1}=\begin{bmatrix}\big(S_r^{(i-1)}\big)^{-1}&-\big(S_r^{(i-1)}\big)^{-1}S_{12}^{(i)}(s_r^{(i)})^{-1}\\\mathbf{0}&(s_r^{(i)})^{-1}\end{bmatrix}$, and $L_r^{(i)}=\begin{bmatrix}L_r^{(i-1)}&l_r^{(i)}\end{bmatrix}$.\label{step8}\\
    Increment $i\gets i+1$.
  }
  \uElse{
    Set $w_r^{(i)} = \begin{bmatrix}w_{\mathrm{re}}^{(i)}&w_{\mathrm{im}}^{(i)}\end{bmatrix} = \begin{bmatrix} \mathrm{Re}(v_i\hat{Z}) & \mathrm{Im}(v_i\hat{Z}) \end{bmatrix}$, $b_r^{(i)}=\begin{bmatrix}b_{\mathrm{re}}^{(i)}\\b_{\mathrm{im}}^{(i)}\end{bmatrix}=\big(w_r^{(i)}\big)^\top B$, $s_r^{(i)}=\begin{bmatrix} 
-\mathrm{Re}(\alpha_i)I_{pm} & -\mathrm{Im}(\alpha_i)I_{pm} \\ 
\mathrm{Im}(\alpha_i)I_{pm} & -\mathrm{Re}(\alpha_i) I_{pm}
\end{bmatrix}$, $l_r^{(i)}=\begin{bmatrix}-I_{pm}&\mathbf{0}\end{bmatrix}$, and $S_{12}^{(i)}=\big(\hat{C}_r^{(i-1)}\big)^\top\hat{Z}l_r^{(i)}+Q_r^{(i-1)}\big(S_r^{(i-1)}\big)^{-\top}B_r^{(i-1)}R^{-1}\big(b_r^{(i)}\big)^\top$.\\
    Compute $x_r^{(i)}$ from \eqref{eq:x_complex_blocks}-\eqref{eq:x_complex}, set $q_r^{(i)}=\big(x_r^{(i)}\big)^{-1}$, and $\hat{c}_r^{(i)}=\Big(l_r^{(i)}-L_r^{(i-1)}\big(S_r^{(i-1)}\big)^{-1}S_{12}^{(i)}\Big)(s_r^{(i)})^{-1}q_r^{(i)}$. \\
    Execute Steps \ref{step7} and \ref{step8}.\\
    Increment $i\gets i+2$.
  }
  
}
Set $\hat{K}^{(i)}=\big(K_1^{(i)}\big)^{-1}K_2^{(i)}$.
\end{algorithm}
\subsection{Automatic Shift Generation}
D-RADI is a recursive interpolation-based model order reduction (MOR) algorithm. The error analysis of MOR algorithms that satisfy Sylvester equations of the forms \eqref{sylv_W} and \eqref{radi_sylv2} has been carried out extensively in \cite{wolfthesis} and \cite{panzerthesis}; we follow it in this subsection.

Define $G_{\perp}^{(i)}(z)$ as follows:
\begin{align}
G_{\perp}^{(i)}(z)=\hat{C}_{\perp}^{(i)}(zE-\hat{A})^{-1}B.
\end{align}
The error $G(z)-G_r^{(i)}(z)$ factorizes as
\begin{align}
G(z)-G_r^{(i)}(z)&=\Big(-\hat{C}_r^{(i)}\big(zI-\hat{A}_r^{(i)}\big)^{-1}(L_r^{(i)})^\top+I\Big)G_{\perp}^{(i)}(z);
\end{align}
cf. \cite{wolfthesis}.

Lemma 3.3 of \cite{wolfthesis} shows that this error is zero when all the poles of the realization $(E,\hat{A},B,\hat{C}_{\perp}^{(i)})$ are unobservable. The $\mathcal{H}_2$ norm error is dominated by the most controllable and observable poles of $G(z)$. If $G_r^{(i)}(z)$ interpolates $G(z)$ at the reciprocals of these poles, the error decreases rapidly \cite{bunse2010h2,gugercin2008h_2}. Although $G(z)$ and $G_{\perp}^{(i)}(z)$ have the same poles, the poles already captured by $G_r^{(i)}(z)$ become weakly observable or unobservable in $G_{\perp}^{(i)}(z)$. Therefore, the mirror images of the reciprocals of the dominant poles of $G_{\perp}^{(i)}(z)$ can be used as ADI shifts in D-RADI to approximate $G(z)$ accurately, since D-RADI interpolates at the mirror images of the ADI shifts.

In large-scale settings, computing the poles of $G_{\perp}^{(i)}(z)$ is infeasible. Instead, Ritz values of $E^{-1}\hat{A}$ can be computed by projecting onto the interpolation basis accumulated during the D-RADI iterations.

Set 
\[
W_{\mathrm{proj}} = \mathrm{orth}\Big(\begin{bmatrix} w_r^{(1)} & \cdots & w_r^{(i)} \end{bmatrix}\Big)
\]
with implicit restart: if the number of columns of $W_{\mathrm{proj}}$ exceeds a prescribed limit, the previous vectors are discarded and a new basis is accumulated. Such restart mechanisms are commonly used in eigenvalue algorithms to limit the basis dimension \cite{saad2011numerical}. Then define the projected matrices:
\[
E_{\mathrm{proj}} = W_{\mathrm{proj}}^\top E W_{\mathrm{proj}}, \quad
A_{\mathrm{proj}} = W_{\mathrm{proj}}^\top \hat{A} W_{\mathrm{proj}}, \quad
C_{\mathrm{proj}} = \hat{C}_{\perp}^{(i)} W_{\mathrm{proj}}.
\]
Compute the eigenvalue decomposition of $A_{\mathrm{proj}}E_{\mathrm{proj}}^{-1}$:
\[
A_{\mathrm{proj}}E_{\mathrm{proj}}^{-1} = T\,\mathrm{diag}\big(\lambda_1,\dots,\lambda_r\big)T^{-1}.
\]
Define $r_{c,l} = C_{\mathrm{proj}} T(:,l)$. The most observable pole $\lambda_{\mathrm{dom}}$ of $A_{\mathrm{proj}}E_{\mathrm{proj}}^{-1}$ is the pole $\lambda_l$ corresponding to the largest normalized residue (a standard measure in modal analysis \cite{gawronski2004dynamics})
\[
\phi_l = \frac{\|r_{c,l}\|_2^2}{1-|\lambda_l|^2}.
\]
The next shift is
\[
\alpha_{i+1}=-\frac{1}{\lambda_{\mathrm{dom}}}.
\]
If all the poles $\lambda_l$ lie outside the unit circle, the previous shift is repeated. Thus, after initialization with an arbitrary shift, D-RADI generates subsequent shifts without user intervention.
\section{Numerical Results}
In this section, the numerical performance of D-RADI is assessed on one moderate-order DARE, allowing comparison with MATLAB's \texttt{idare}. Then, two large-scale DAREs of order $10^6$ are considered, which are prohibitively expensive for MATLAB's \texttt{idare} but feasible for D-RADI. MATLAB codes for reproducing the results in this section are publicly available at \cite{mycode}. The first shift is set to $\alpha_1=2$ in all experiments. All experiments are performed in MATLAB R2025b on a Windows 11 laptop with 32 GB RAM and an Intel(R) Core(TM) Ultra 9 285H processor running at 2.9 GHz.

\subsection{Example 1: Comparison with MATLAB's \texttt{idare} Function}
The matrices $E\in\mathbb{R}^{200\times 200}$, $A\in\mathbb{R}^{200\times 200}$, $B\in\mathbb{R}^{200\times 1}$, and $C_1\in\mathbb{R}^{1\times200}$ are the state-space matrices of the heat transfer model taken from the benchmark collection in \cite{chahlaoui2005benchmark}. The matrices $C_2$, $R$, and $Z$ are set as
\[
C_2=\begin{bmatrix}1&\mathbf{0}_{1\times199}\end{bmatrix},\quad R=-0.1,\quad \text{and}\quad Z=-0.5.
\]
The DARE \eqref{dare} is solved using MATLAB's \texttt{idare} routine, and the numerical rank of $Q$, checked using MATLAB's \texttt{rank} command, was $66$. The tolerance $\tau$ in D-RADI is set to MATLAB's floating-point relative accuracy \texttt{eps}. The maximum number of columns for $W_{\mathrm{proj}}$ is set to $20$. D-RADI converged in the $43^{\text{rd}}$ iteration. The normalized residual history
\[
\dfrac{\|\hat{A}^\top \hat{Q}^{(i)} \hat{A}-E^\top \hat{Q}^{(i)} E-\hat{A}^\top \hat{Q}^{(i)} B(B^\top \hat{Q}^{(i)} B+R)^{-1}B^\top \hat{Q}^{(i)}\hat{A}+\hat{C}^\top\hat{Z}\hat{C}\|_2}{\|\hat{C}^\top\hat{Z}\hat{C}\|_2}
\]
is plotted in Figure \ref{fig0}.
\begin{figure}[!h]
  \centering
  \includegraphics[width=10cm]{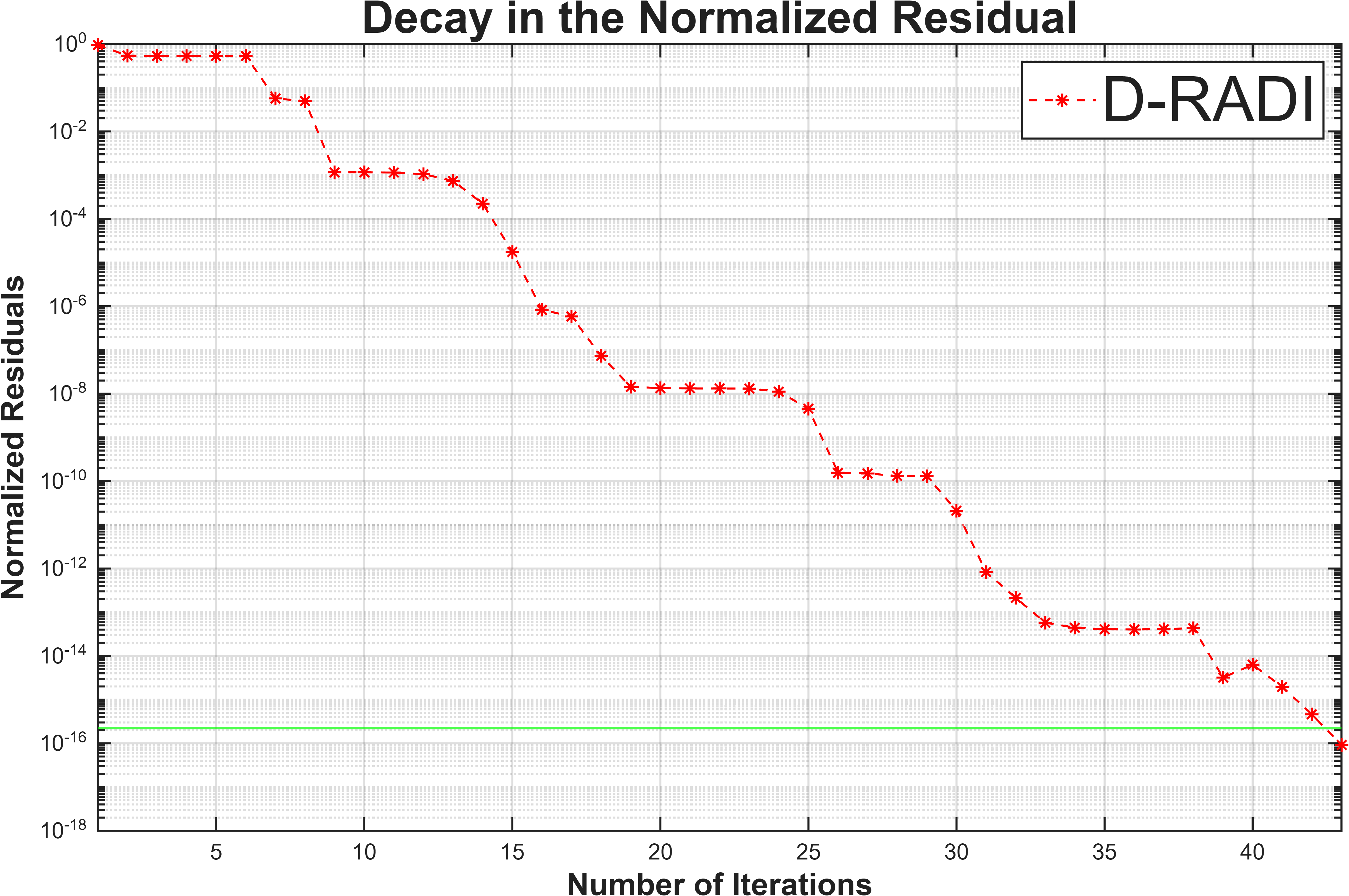}
  \caption{Decay of the normalized residual}
  \label{fig0}
\end{figure}
The relative differences between $Q$ and $K$ computed by MATLAB's \texttt{idare} routine and D-RADI are
\[
\dfrac{\|Q^{(\text{idare})}-Q^{(\text{d-radi})}\|_2}{\|Q^{(\text{idare})}\|_2}= 4.9627\times10^{-12}\quad\text{and}\quad \dfrac{\|K^{(\text{idare})}-K^{(\text{d-radi})}\|_2}{\|K^{(\text{idare})}\|_2}=1.0304\times10^{-12}.
\]
The $66$ nonzero singular values of $Q^{(\text{idare})}$ are compared with those of $Q^{(\text{d-radi})}$ in Figure \ref{fig1}.
\begin{figure}[!h]
  \centering
  \includegraphics[width=10cm]{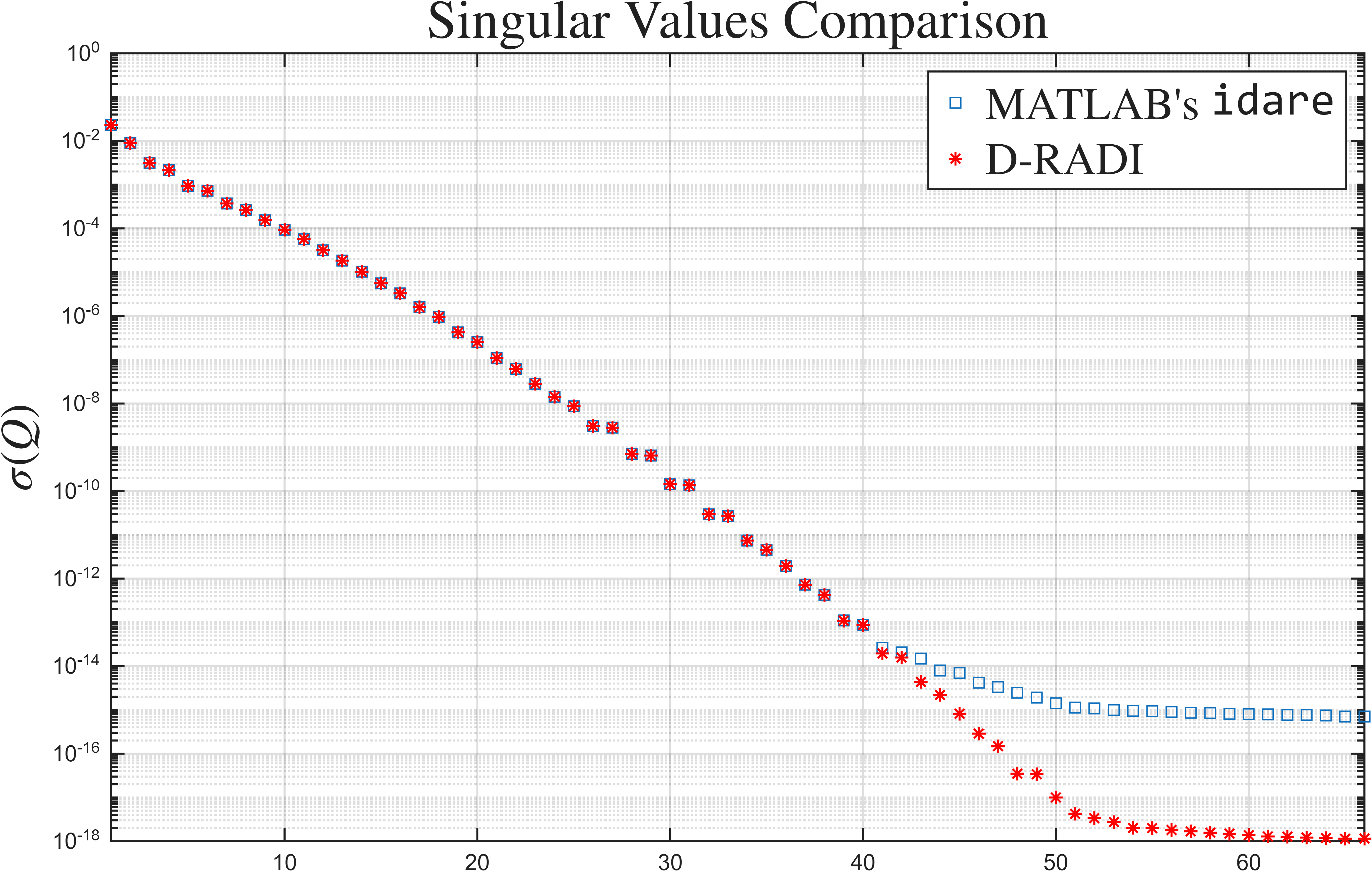}
  \caption{Singular Values Comparison}
  \label{fig1}
\end{figure}
These results show that D-RADI produces an accurate approximation.

\subsection{Example 2: Large-scale DARE}
The matrices $E\in\mathbb{R}^{10^6\times 10^6}$, $A\in\mathbb{R}^{10^6\times 10^6}$, $B\in\mathbb{R}^{10^6\times 1}$, and $\hat{C}\in\mathbb{R}^{2\times10^6}$ are the state-space matrices of the discrete-time procedural model constructed as described in \cite{zulfiqar2026}. The poles $\lambda(E^{-1}A)$ are plotted in Figure \ref{fig2}.
\begin{figure}[!h]
  \centering
  \includegraphics[width=10cm]{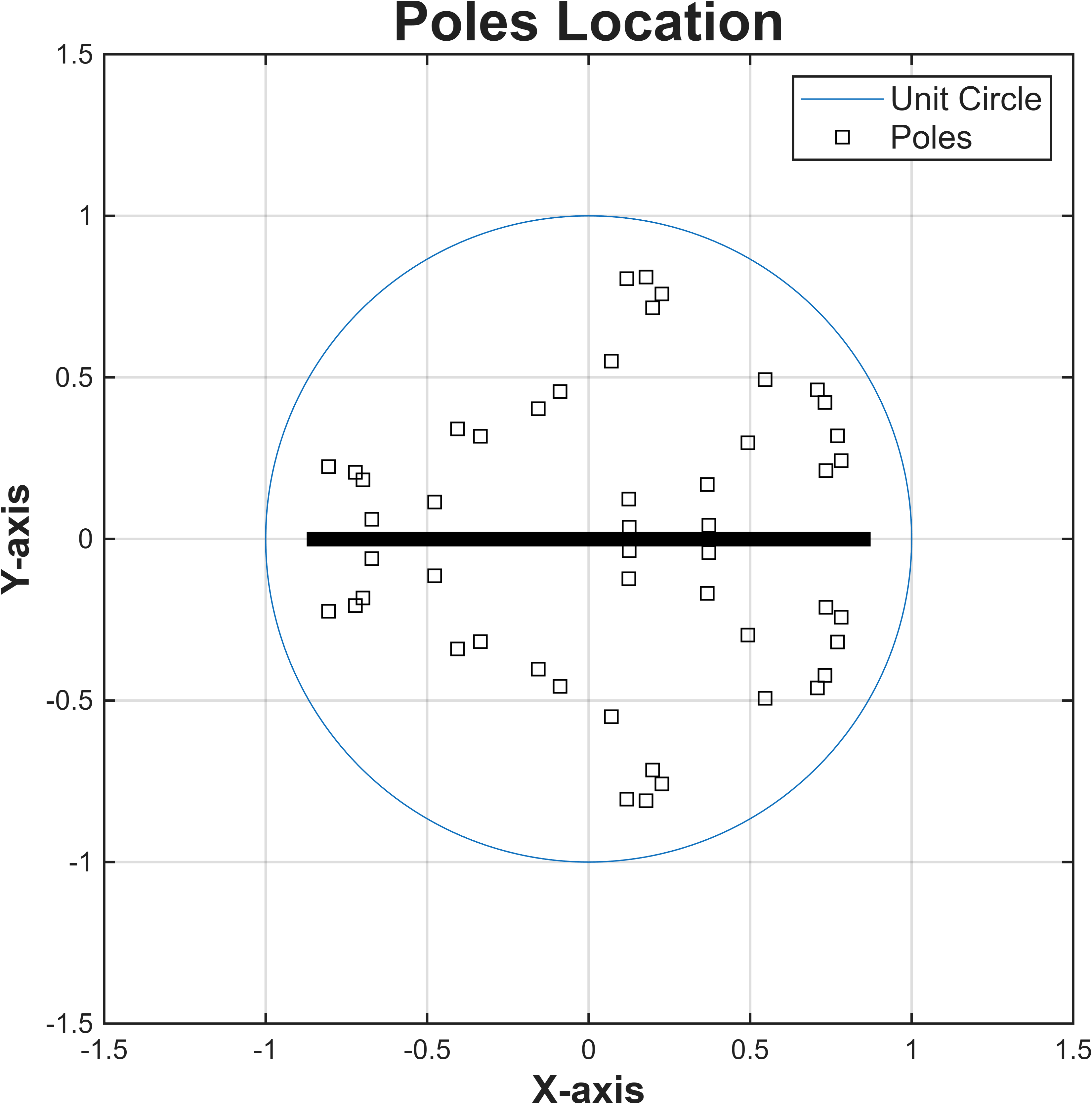}
  \caption{Pole locations}
  \label{fig2}
\end{figure}
The matrices $R$ and $Z$ are set as
\[
R = -0.0431\quad\text{and}\quad Z =-0.6045.
\]
The tolerance $\tau$ in D-RADI is set to $10^{-10}$. The maximum number of columns for $W_{\mathrm{proj}}$ is set to $10$. D-RADI converged in the $19^{\text{th}}$ iteration in $24.4192$ seconds. The normalized residual history is plotted in Figure \ref{fig3}.
\begin{figure}[!h]
  \centering
  \includegraphics[width=10cm]{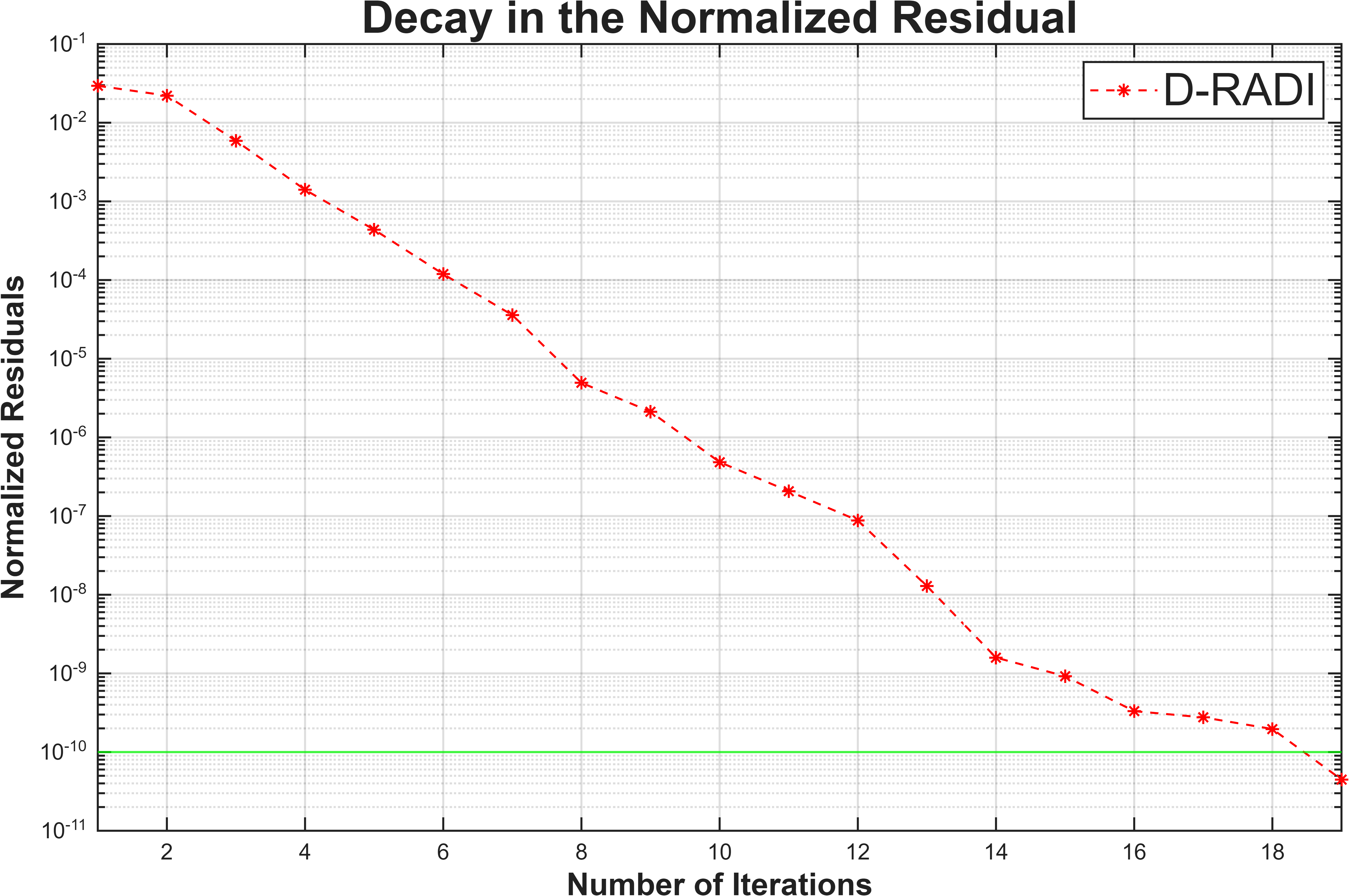}
  \caption{Decay of the normalized residual}
  \label{fig3}
\end{figure}
It can be seen that D-RADI produces an accurate approximation in under 25 seconds.

\subsection{Example 3: Large-scale DARE}
The matrices $E\in\mathbb{R}^{10^6\times 10^6}$, $A\in\mathbb{R}^{10^6\times 10^6}$, $B\in\mathbb{R}^{10^6\times 2}$, and $\hat{C}\in\mathbb{R}^{4\times10^6}$ are the state-space matrices of the discrete-time procedural model constructed as described in \cite{zulfiqar2026}. The poles $\lambda(E^{-1}A)$ are plotted in Figure \ref{fig4}.
\begin{figure}[!h]
  \centering
  \includegraphics[width=10cm]{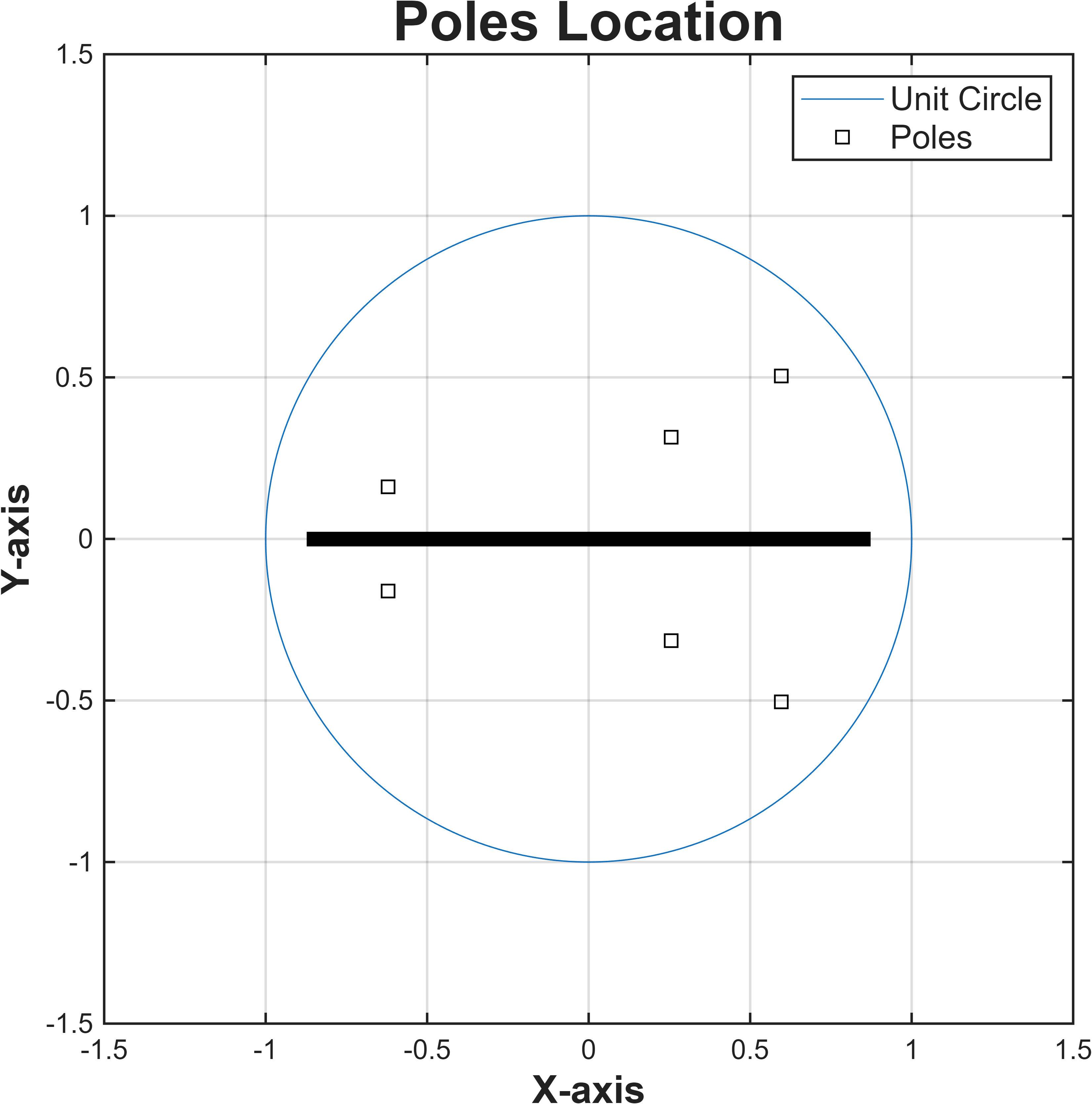}
  \caption{Pole locations}
  \label{fig4}
\end{figure}
The matrices $R$ and $Z$ are set as
\[
R = \begin{bmatrix}
0.1733&0.7136\\
0.7136&0.7243
\end{bmatrix}\quad\text{and}\quad Z =\begin{bmatrix}
0.5256&0.9860\\
0.9860&0.4559
\end{bmatrix}.
\]
The tolerance $\tau$ in D-RADI is set to $10^{-10}$. The maximum number of columns for $W_{\mathrm{proj}}$ is set to $20$. D-RADI converged in the $12^{\text{th}}$ iteration in $20.1618$ seconds. The normalized residual history is plotted in Figure \ref{fig5}.
\begin{figure}[!h]
  \centering
  \includegraphics[width=10cm]{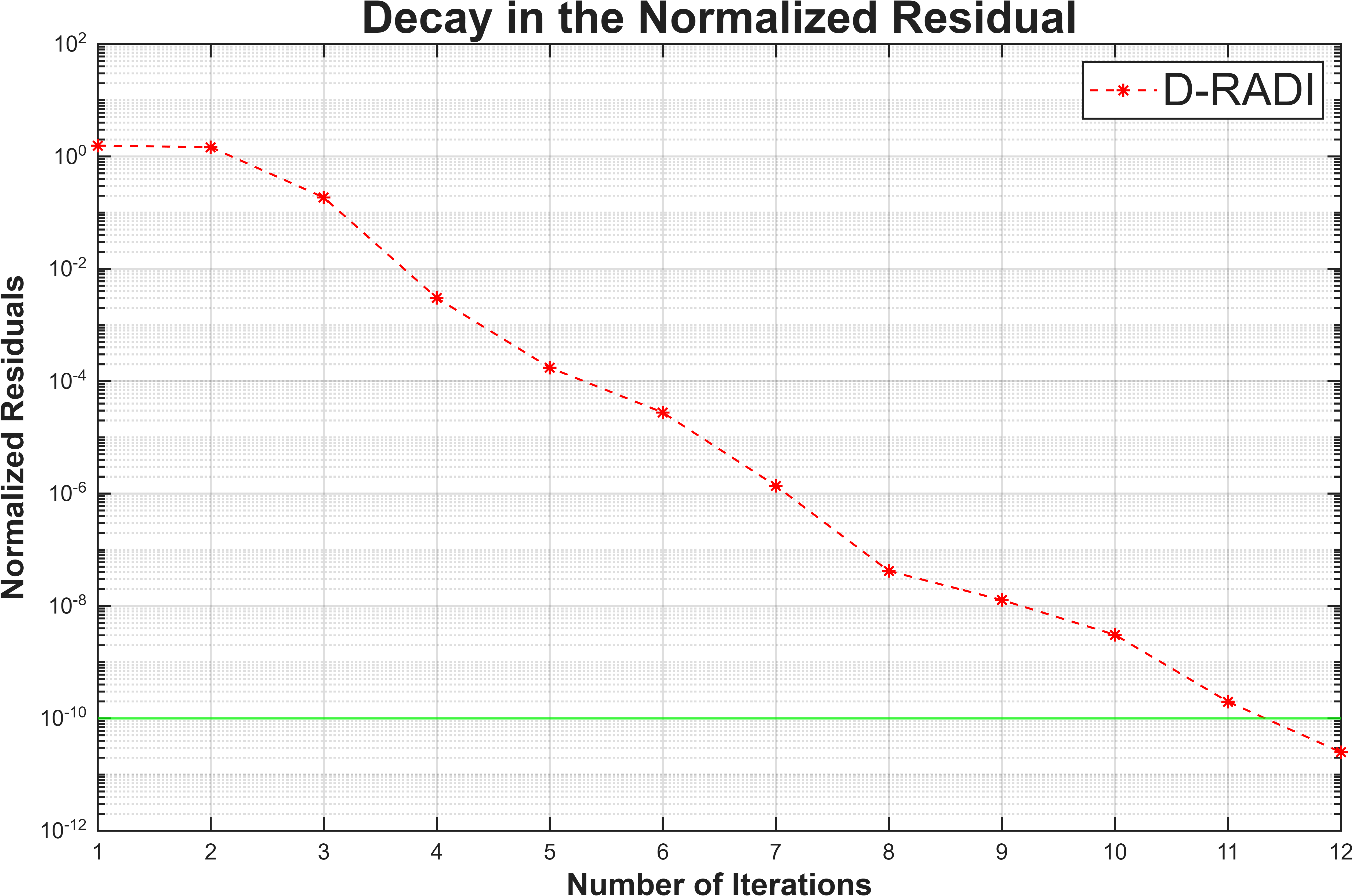}
  \caption{Decay of the normalized residual}
  \label{fig5}
\end{figure}
It can be seen that D-RADI produces an accurate approximation in about 20 seconds.
\section{Conclusion}
This paper addresses the numerical solution of large-scale DAREs by proposing a low-rank ADI algorithm, named D-RADI. D-RADI exploits the interpolatory nature of low-rank ADI algorithms. Despite being a projection-based method, D-RADI uses its pole placement property to ensure that the projected DARE admits a stabilizing solution throughout the ADI iterations. We also propose an efficient approach for automatic shift generation. This makes D-RADI fully autonomous: starting from an arbitrary shift, it generates subsequent shifts without user intervention until the residual reaches a prescribed tolerance. Numerical experiments show that D-RADI matches the accuracy of MATLAB's \texttt{idare} on moderate-order DAREs, while efficiently solving large-scale DAREs that are infeasible for \texttt{idare}. Overall, D-RADI is an effective and autonomous solver for computing low-rank solutions of large-scale DAREs.
\section*{Appendix}
\begin{proof}
1. From \eqref{lyap_x}, we get
\begin{equation}\label{eq:pf_stein2}
X_r^{(i)}+B_r^{(i)}R^{-1}(B_r^{(i)})^\top
=
(S_r^{(i)})^\top X_r^{(i)} S_r^{(i)}
-
(L_r^{(i)})^\top \hat Z L_r^{(i)}.
\end{equation}
For compactness, define
\[
M_r^{(i)} := (L_r^{(i)})^\top \hat Z L_r^{(i)},
\qquad
\Pi_r^{(i)} := X_r^{(i)}+B_r^{(i)}R^{-1}(B_r^{(i)})^\top .
\]
Introduce
\[
\Delta_r^{(i)}
:=
X_r^{(i)}
-
(S_r^{(i)})^{-\top}M_r^{(i)}(S_r^{(i)})^{-1}.
\]
It follows that
\[
\Pi_r^{(i)}
=
(S_r^{(i)})^\top \Delta_r^{(i)} S_r^{(i)},
\]
and hence
\[
(\Pi_r^{(i)})^{-1}
=
(S_r^{(i)})^{-1}(\Delta_r^{(i)})^{-1}(S_r^{(i)})^{-\top}.
\]

By the choice of the free parameter in the theorem,
\[
\hat C_r^{(i)}
=
L_r^{(i)}(S_r^{(i)})^{-1}(X_r^{(i)})^{-1}
=
L_r^{(i)}(S_r^{(i)})^{-1}Q_r^{(i)}.
\]
Therefore,
\[
(L_r^{(i)})^\top \hat Z \hat C_r^{(i)}
=
M_r^{(i)}(S_r^{(i)})^{-1}Q_r^{(i)}.
\]
Using the assumed identity
\[
\hat A_r^{(i)}
=
(S_r^{(i)})^\top
-
(L_r^{(i)})^\top \hat Z \hat C_r^{(i)},
\]
we obtain
\[
\hat A_r^{(i)}
=
(S_r^{(i)})^\top
-
M_r^{(i)}(S_r^{(i)})^{-1}Q_r^{(i)}.
\]
Furthermore,
\[
\Delta_r^{(i)}Q_r^{(i)}
=
X_r^{(i)}Q_r^{(i)}
-
(S_r^{(i)})^{-\top}M_r^{(i)}(S_r^{(i)})^{-1}Q_r^{(i)}
=
I
-
(S_r^{(i)})^{-\top}M_r^{(i)}(S_r^{(i)})^{-1}Q_r^{(i)}.
\]
Consequently,
\[
\hat A_r^{(i)}
=
(S_r^{(i)})^\top \Delta_r^{(i)}Q_r^{(i)}.
\]
Thus,
\begin{equation}\label{eq:pf_key}
(\Pi_r^{(i)})^{-1}\hat A_r^{(i)}
=
(S_r^{(i)})^{-1}(\Delta_r^{(i)})^{-1}(S_r^{(i)})^{-\top}
(S_r^{(i)})^\top \Delta_r^{(i)}Q_r^{(i)}
=
(S_r^{(i)})^{-1}Q_r^{(i)}.
\end{equation}

We first verify that \(Q_r^{(i)}=(X_r^{(i)})^{-1}\) satisfies the projected DARE \eqref{proj_dare}.

Since \(X_r^{(i)}\) and \(R+(B_r^{(i)})^\top (X_r^{(i)})^{-1}B_r^{(i)}\) are invertible, \(\Pi_r^{(i)}\) is invertible. Furthermore, because \(S_r^{(i)}\) is invertible and \(\Pi_r^{(i)}=(S_r^{(i)})^\top \Delta_r^{(i)} S_r^{(i)}\), the matrix \(\Delta_r^{(i)}\) is also invertible. By the Woodbury identity \cite{golub2013matrix},
\begin{equation}\label{eq:pf_woodbury}
(\Pi_r^{(i)})^{-1}
=
Q_r^{(i)}
-
Q_r^{(i)}B_r^{(i)}
\Bigl(R+(B_r^{(i)})^\top Q_r^{(i)}B_r^{(i)}\Bigr)^{-1}
(B_r^{(i)})^\top Q_r^{(i)}.
\end{equation}
Hence the residual of \eqref{proj_dare} can be written as
\[
\begin{aligned}
\mathcal R_r^{(i)}
&=
(\hat A_r^{(i)})^\top Q_r^{(i)}\hat A_r^{(i)}
-
Q_r^{(i)}-
(\hat A_r^{(i)})^\top Q_r^{(i)}B_r^{(i)}
\Bigl((B_r^{(i)})^\top Q_r^{(i)}B_r^{(i)}+R\Bigr)^{-1}
(B_r^{(i)})^\top Q_r^{(i)}\hat A_r^{(i)}+
(\hat C_r^{(i)})^\top \hat Z \hat C_r^{(i)} \\
&=
(\hat A_r^{(i)})^\top(\Pi_r^{(i)})^{-1}\hat A_r^{(i)}
-
Q_r^{(i)}
+
(\hat C_r^{(i)})^\top \hat Z \hat C_r^{(i)}.
\end{aligned}
\]
Using \eqref{eq:pf_key} and the symmetry of \(M_r^{(i)}\) and \(Q_r^{(i)}\), we find
\[
\begin{aligned}
(\hat A_r^{(i)})^\top(\Pi_r^{(i)})^{-1}\hat A_r^{(i)}
&=
(\hat A_r^{(i)})^\top (S_r^{(i)})^{-1}Q_r^{(i)} \\
&=
\Bigl(
S_r^{(i)}
-
Q_r^{(i)}(S_r^{(i)})^{-\top}M_r^{(i)}
\Bigr)
(S_r^{(i)})^{-1}Q_r^{(i)} \\
&=
Q_r^{(i)}
-
Q_r^{(i)}(S_r^{(i)})^{-\top}M_r^{(i)}(S_r^{(i)})^{-1}Q_r^{(i)}.
\end{aligned}
\]
Moreover,
\[
\begin{aligned}
(\hat C_r^{(i)})^\top \hat Z \hat C_r^{(i)}
&=
Q_r^{(i)}(S_r^{(i)})^{-\top}
(L_r^{(i)})^\top \hat Z L_r^{(i)}
(S_r^{(i)})^{-1}Q_r^{(i)} \\
&=
Q_r^{(i)}(S_r^{(i)})^{-\top}
M_r^{(i)}
(S_r^{(i)})^{-1}Q_r^{(i)}.
\end{aligned}
\]
Substituting these two identities into the expression for \(\mathcal R_r^{(i)}\) yields
\[
\mathcal R_r^{(i)}=0.
\]
Therefore, \(Q_r^{(i)}=(X_r^{(i)})^{-1}\) is a solution of the projected DARE \eqref{proj_dare}.

It remains to prove that this solution is stabilizing. Since
\[
\hat A_r^{(i)}
=
A_r^{(i)}-B_r^{(i)}R^{-1}C_{2,r}^{(i)},
\]
we have
\[
\begin{aligned}
&(B_r^{(i)})^\top Q_r^{(i)}A_r^{(i)}+C_{2,r}^{(i)} \\
&\qquad=
(B_r^{(i)})^\top Q_r^{(i)}\hat A_r^{(i)}
+
(B_r^{(i)})^\top Q_r^{(i)}B_r^{(i)}R^{-1}C_{2,r}^{(i)}
+
C_{2,r}^{(i)} \\
&\qquad=
(B_r^{(i)})^\top Q_r^{(i)}\hat A_r^{(i)}
+
\Bigl((B_r^{(i)})^\top Q_r^{(i)}B_r^{(i)}+R\Bigr)
R^{-1}C_{2,r}^{(i)}.
\end{aligned}
\]
Thus the feedback matrix in the theorem can be rewritten as
\[
\begin{aligned}
&\Bigl((B_r^{(i)})^\top Q_r^{(i)}B_r^{(i)}+R\Bigr)^{-1}
\Bigl((B_r^{(i)})^\top Q_r^{(i)}A_r^{(i)}+C_{2,r}^{(i)}\Bigr) \\
&\qquad=
R^{-1}C_{2,r}^{(i)}
+
\Bigl((B_r^{(i)})^\top Q_r^{(i)}B_r^{(i)}+R\Bigr)^{-1}
(B_r^{(i)})^\top Q_r^{(i)}\hat A_r^{(i)}.
\end{aligned}
\]
Therefore,
\[
\begin{aligned}
A_{r,cl}^{(i)}
&=
A_r^{(i)}
-
B_r^{(i)}
\Bigl((B_r^{(i)})^\top Q_r^{(i)}B_r^{(i)}+R\Bigr)^{-1}
\Bigl((B_r^{(i)})^\top Q_r^{(i)}A_r^{(i)}+C_{2,r}^{(i)}\Bigr) \\
&=
\hat A_r^{(i)}
-
B_r^{(i)}
\Bigl((B_r^{(i)})^\top Q_r^{(i)}B_r^{(i)}+R\Bigr)^{-1}
(B_r^{(i)})^\top Q_r^{(i)}\hat A_r^{(i)}.
\end{aligned}
\]
Using the Woodbury identity \cite{golub2013matrix} in the form
\[
I
-
B_r^{(i)}
\Bigl((B_r^{(i)})^\top Q_r^{(i)}B_r^{(i)}+R\Bigr)^{-1}
(B_r^{(i)})^\top Q_r^{(i)}
=
X_r^{(i)}(\Pi_r^{(i)})^{-1},
\]
and then applying \eqref{eq:pf_key}, we obtain
\[
A_{r,cl}^{(i)}
=
X_r^{(i)}(\Pi_r^{(i)})^{-1}\hat A_r^{(i)}
=
X_r^{(i)}(S_r^{(i)})^{-1}Q_r^{(i)}
=
X_r^{(i)}(S_r^{(i)})^{-1}(X_r^{(i)})^{-1}.
\]
Hence \(A_{r,cl}^{(i)}\) is similar to \((S_r^{(i)})^{-1}\). Since
\[
\lambda\bigl(S_r^{(i)}\bigr)=\{-\alpha_1,\ldots,-\alpha_i\},
\]
with each eigenvalue having multiplicity \(p+m\), it follows that
\[
\lambda\bigl(A_{r,cl}^{(i)}\bigr)
=
\lambda\bigl((S_r^{(i)})^{-1}\bigr)
=
\left\{
-\frac{1}{\alpha_1},\ldots,-\frac{1}{\alpha_i}
\right\},
\]
where each eigenvalue \(-1/\alpha_j\) has multiplicity \(p+m\). Because \(|\alpha_j|>1\) for all \(j=1,\ldots,i\), we have
\[
\left|-\frac{1}{\alpha_j}\right|<1.
\]
Thus all eigenvalues of \(A_{r,cl}^{(i)}\) lie strictly inside the unit circle. Consequently, \(Q_r^{(i)}=(X_r^{(i)})^{-1}\) is a stabilizing solution of the projected DARE \eqref{proj_dare}.

2. Recall that
\[
\Pi_r^{(i)}
=
X_r^{(i)}
+
B_r^{(i)}R^{-1}(B_r^{(i)})^\top .
\]
By the Woodbury identity,
\[
Q_r^{(i)}
-
Q_r^{(i)}B_r^{(i)}
\Bigl((B_r^{(i)})^\top Q_r^{(i)}B_r^{(i)}+R\Bigr)^{-1}
(B_r^{(i)})^\top Q_r^{(i)}
=
(\Pi_r^{(i)})^{-1}.
\]
Hence the residual can be written as
\begin{equation}
R_q^{(i)}
=
\hat{A}^\top W_r^{(i)}
(\Pi_r^{(i)})^{-1}
(W_r^{(i)})^\top \hat{A}
-
E^\top W_r^{(i)}Q_r^{(i)}(W_r^{(i)})^\top E
+
\hat{C}^\top \hat{Z}\hat{C}.
\label{eq:rq_Pi_part2}
\end{equation}

Using \eqref{sylv_W} and the definition
\[
\hat{C}_{\perp}^{(i)}
=
\hat{C}
-
\hat{C}_r^{(i)}(W_r^{(i)})^\top E,
\]
we have
\begin{equation}
\hat{A}^\top W_r^{(i)}
=
E^\top W_r^{(i)}(\hat{A}_r^{(i)})^\top
-
(\hat{C}_{\perp}^{(i)})^\top \hat{Z}L_r^{(i)}
\label{eq:consistency_part2}
\end{equation}
and therefore
\begin{equation}
(W_r^{(i)})^\top \hat{A}
=
\hat{A}_r^{(i)}(W_r^{(i)})^\top E
-
(L_r^{(i)})^\top \hat{Z}\hat{C}_{\perp}^{(i)} .
\label{eq:consistency_transpose_part2}
\end{equation}
Moreover,
\[
\hat{C}
=
\hat{C}_{\perp}^{(i)}
+
\hat{C}_r^{(i)}(W_r^{(i)})^\top E.
\]
Substituting \eqref{eq:consistency_part2} and \eqref{eq:consistency_transpose_part2} into \eqref{eq:rq_Pi_part2}, expanding, and grouping the terms gives
\begin{equation}
\begin{aligned}
R_q^{(i)}
&=
E^\top W_r^{(i)}
\Bigl[
(\hat{A}_r^{(i)})^\top
(\Pi_r^{(i)})^{-1}
\hat{A}_r^{(i)}
-
Q_r^{(i)}
+
(\hat{C}_r^{(i)})^\top \hat{Z}\hat{C}_r^{(i)}
\Bigr]
(W_r^{(i)})^\top E\\
&\quad
+
E^\top W_r^{(i)}
\Bigl[
(\hat{C}_r^{(i)})^\top
-
(\hat{A}_r^{(i)})^\top
(\Pi_r^{(i)})^{-1}
(L_r^{(i)})^\top
\Bigr]
\hat{Z}\hat{C}_{\perp}^{(i)}\\
&\quad
+
(\hat{C}_{\perp}^{(i)})^\top \hat{Z}
\Bigl[
\hat{C}_r^{(i)}
-
L_r^{(i)}
(\Pi_r^{(i)})^{-1}
\hat{A}_r^{(i)}
\Bigr]
(W_r^{(i)})^\top E\\
&\quad
+
(\hat{C}_{\perp}^{(i)})^\top \hat{Z}
\Bigl[
I_{pm}
+
L_r^{(i)}
(\Pi_r^{(i)})^{-1}
(L_r^{(i)})^\top \hat{Z}
\Bigr]
\hat{C}_{\perp}^{(i)} .
\end{aligned}
\label{eq:rq_collected_part2}
\end{equation}
The bracket in the first line is the residual of the projected DARE \eqref{proj_dare}, written using the Woodbury identity, and therefore vanishes.

We now simplify the cross-terms. From \eqref{lyap_x} and
\[
L_r^{(i)}
=
\hat{C}_r^{(i)}X_r^{(i)}S_r^{(i)},
\]
we obtain
\[
\Pi_r^{(i)}
=
\hat{A}_r^{(i)}X_r^{(i)}S_r^{(i)} .
\]
Thus
\[
(\Pi_r^{(i)})^{-1}\hat{A}_r^{(i)}
=
(S_r^{(i)})^{-1}Q_r^{(i)} .
\]
Consequently,
\[
L_r^{(i)}
(\Pi_r^{(i)})^{-1}
\hat{A}_r^{(i)}
=
\hat{C}_r^{(i)}X_r^{(i)}S_r^{(i)}
(S_r^{(i)})^{-1}Q_r^{(i)}
=
\hat{C}_r^{(i)}X_r^{(i)}Q_r^{(i)}
=
\hat{C}_r^{(i)} .
\]
Since \(\Pi_r^{(i)}\) is symmetric, taking the transpose yields
\[
(\hat{A}_r^{(i)})^\top
(\Pi_r^{(i)})^{-1}
(L_r^{(i)})^\top
=
(\hat{C}_r^{(i)})^\top .
\]
Therefore the two cross-terms in \eqref{eq:rq_collected_part2} vanish, and
\begin{equation}
R_q^{(i)}
=
(\hat{C}_{\perp}^{(i)})^\top \hat{Z}
\Bigl[
I_{pm}
+
L_r^{(i)}
(\Pi_r^{(i)})^{-1}
(L_r^{(i)})^\top \hat{Z}
\Bigr]
\hat{C}_{\perp}^{(i)} .
\label{eq:rq_middle_part2}
\end{equation}

It remains to simplify the middle matrix. Using again
\[
L_r^{(i)}
=
\hat{C}_r^{(i)}X_r^{(i)}S_r^{(i)},
\]
we have
\[
\Pi_r^{(i)}
=
(S_r^{(i)})^\top X_r^{(i)}
\Bigl[
I
-
(\hat{C}_r^{(i)})^\top \hat{Z}\hat{C}_r^{(i)}X_r^{(i)}
\Bigr]
S_r^{(i)} .
\]
Hence
\[
L_r^{(i)}
(\Pi_r^{(i)})^{-1}
(L_r^{(i)})^\top
=
\hat{C}_r^{(i)}X_r^{(i)}
\Bigl[
I
-
(\hat{C}_r^{(i)})^\top \hat{Z}\hat{C}_r^{(i)}X_r^{(i)}
\Bigr]^{-1}
(\hat{C}_r^{(i)})^\top .
\]
Define
\[
M^{(i)}
:=
I_{pm}
-
\hat{C}_r^{(i)}X_r^{(i)}
(\hat{C}_r^{(i)})^\top \hat{Z}.
\]
Direct multiplication shows that
\[
M^{(i)}
\hat{C}_r^{(i)}X_r^{(i)}
\Bigl[
I
-
(\hat{C}_r^{(i)})^\top \hat{Z}\hat{C}_r^{(i)}X_r^{(i)}
\Bigr]^{-1}
(\hat{C}_r^{(i)})^\top
=
\hat{C}_r^{(i)}X_r^{(i)}
(\hat{C}_r^{(i)})^\top .
\]
Thus
\[
L_r^{(i)}
(\Pi_r^{(i)})^{-1}
(L_r^{(i)})^\top
=
(M^{(i)})^{-1}
\hat{C}_r^{(i)}X_r^{(i)}
(\hat{C}_r^{(i)})^\top .
\]
Since
\[
\hat{C}_r^{(i)}X_r^{(i)}
(\hat{C}_r^{(i)})^\top \hat{Z}
=
I_{pm}-M^{(i)},
\]
we obtain
\[
\begin{aligned}
I_{pm}
+
L_r^{(i)}
(\Pi_r^{(i)})^{-1}
(L_r^{(i)})^\top \hat{Z}
&=
I_{pm}
+
(M^{(i)})^{-1}
\hat{C}_r^{(i)}X_r^{(i)}
(\hat{C}_r^{(i)})^\top \hat{Z}\\
&=
I_{pm}
+
(M^{(i)})^{-1}
(I_{pm}-M^{(i)})\\
&=
(M^{(i)})^{-1}.
\end{aligned}
\]
Substituting this into \eqref{eq:rq_middle_part2} yields
\[
R_q^{(i)}
=
(\hat{C}_{\perp}^{(i)})^\top
\hat{Z}
(M^{(i)})^{-1}
\hat{C}_{\perp}^{(i)} .
\]

Moreover, since
\[
(W_r^{(i)})^\top E V_r^{(i)}=I,
\qquad
\hat{C}V_r^{(i)}=\hat{C}_r^{(i)},
\]
we have
\[
\hat{C}_{\perp}^{(i)}V_r^{(i)}
=
\hat{C}V_r^{(i)}
-
\hat{C}_r^{(i)}(W_r^{(i)})^\top E V_r^{(i)}
=
\hat{C}_r^{(i)}-\hat{C}_r^{(i)}
=
0.
\]
Therefore,
\[
(V_r^{(i)})^\top R_q^{(i)}V_r^{(i)}
=
0.
\]

Finally, from
\[
X_r^{(i)}
=
\begin{bmatrix}
X_r^{(i-1)}&0\\
0&x_r^{(i)}
\end{bmatrix},
\qquad
\hat{C}_r^{(i)}
=
\begin{bmatrix}
\hat{C}_r^{(i-1)}&\hat{c}_r^{(i)}
\end{bmatrix},
\]
we obtain
\[
\hat{C}_r^{(i)}X_r^{(i)}(\hat{C}_r^{(i)})^\top
=
\hat{C}_r^{(i-1)}X_r^{(i-1)}(\hat{C}_r^{(i-1)})^\top
+
\hat{c}_r^{(i)}x_r^{(i)}(\hat{c}_r^{(i)})^\top .
\]
Hence
\[
\begin{aligned}
M^{(i)}
&=
I_{pm}
-
\hat{C}_r^{(i)}X_r^{(i)}(\hat{C}_r^{(i)})^\top \hat{Z}\\
&=
I_{pm}
-
\hat{C}_r^{(i-1)}X_r^{(i-1)}(\hat{C}_r^{(i-1)})^\top \hat{Z}
-
\hat{c}_r^{(i)}x_r^{(i)}(\hat{c}_r^{(i)})^\top \hat{Z}\\
&=
M^{(i-1)}
-
\hat{c}_r^{(i)}x_r^{(i)}(\hat{c}_r^{(i)})^\top \hat{Z},
\end{aligned}
\]
with \(M^{(0)}=I_{pm}\).

3. Consider
\begin{align}
& \hat{A}^\top W_r^{(i)}-E^\top W_r^{(i)} (\hat{A}_r^{(i)})^\top+(\hat{C}_{\perp}^{(i)})^\top \hat{Z}L_r^{(i)}\nonumber\\
&=\hat{A}^\top W_r^{(i)}-E^\top W_r^{(i)} \big(S_r^{(i)}-(\hat{C}_r^{(i)})^\top\hat{Z}L_r^{(i)}\big)+\big(\hat{C}-\hat{C}_r^{(i)}(W_r^{(i)})^\top E\big)^\top \hat{Z}L_r^{(i)}\nonumber\\
&=\hat{A}^\top W_r^{(i)}-E^\top W_r^{(i)}S_r^{(i)}+\hat{C}^\top\hat{Z}L_r^{(i)}\nonumber\\
&=0.\nonumber
\end{align}
This completes the proof.
\end{proof}
%\bibliography{mybibfile}

\end{document}